\documentclass[12pt]{article}
\usepackage{amsmath}
\usepackage{amssymb}
\usepackage{amsthm}
\usepackage{framed}
\usepackage{mathrsfs}
\usepackage[normalem]{ulem}
\usepackage[all]{xy}
\usepackage{multicol}
\usepackage{lipsum}
\usepackage{enumitem}

\theoremstyle{theorem}
\newtheorem{theorem}{Theorem}
\newtheorem{corollary}{Corollary}
 
\theoremstyle{definition}
\newtheorem{definition}{Definition}

\newtheorem{lemma}{Lemma}

   \newcommand{\ee}{\ensuremath{\mathbf{\boldsymbol{\mathrm{e}}_1}}}
   \newcommand{\eee}{\ensuremath{\mathbf{\boldsymbol{\mathrm{e}}_2}}}

   \newcommand{\h}{\ensuremath{\mathbb{H}}}

\newcommand{\z}{\zeta}
\newcommand{\w}{\omega}
\newcommand{\R}{\mathbb{R}}
\newcommand{\Bi}{\mathbb{B}\mathbb{C}}
\newcommand{\C}{\mathbb{C}}

\newcommand{\centre}{\begin{center}\vspace{-2pt}}
\newcommand{\stopcentre}{\vspace{-2pt}\end{center}}

\usepackage{graphicx}
\graphicspath{%
    {converted_graphics/}
    {C:/Bicomplex/converted_graphics/}
}
\begin{document}

\title{Bicomplex Matrices and Operators: Jordan Forms, Invariant Subspace Lattice Diagrams, and Compact Operators}
\author{William Johnston and Rebecca G. Wahl\vspace{-0ex}}
\markright{Bicomplex Linear Transformations}
\maketitle 

\begin{abstract} \noindent {This paper extends topics in linear algebra and operator theory for linear transformations on complex vector spaces to those on bicomplex Hilbert and Banach spaces. For example, Definition 3 for the first time defines a bicomplex vector space, its dimension, and its basis in terms of a corresponding vectorial idempotent representation, and the paper shows how an $n \times n$ bicomplex matrix's idempotent representation leads to its bicomplex Jordan form and a description of its bicomplex invariant subspace lattice diagram. Similarly, the paper rigorously defines for the first time ``bicomplex Banach and Hilbert'' spaces, and then it expands the theory of compact operators on complex Banach and Hilbert spaces to those on bicomplex Banach and Hilbert spaces.  In these ways, the paper shows that complex linear algebra and operator theory are not necessarily built upon the broadest and most natural set of scalars to study. Instead, when using the idempotent representation, such results surprisingly generalize in a straightforward way to the higher-dimensional case of bicomplex values.}
\end{abstract} 

\noindent \section{Preliminaries.} 

\noindent Bicomplex numbers\footnote{This introduction is adapted from \cite{JM}.} are of the form $\z=z_1+jz_2$, where $z_1=x_1+iy_1$ and $z_2=x_2+iy_2$ are complex numbers, $i^2={-1}$, and $j^2={-1}$.\ Corrado Segre first described them in 1892 \cite{Segre}.\ A host of compelling analytic function properties extend outward to bicomplex functional analysis in often straightforward ways.\ These include generalizations of Euler's formula, the representation of complex-analytic functions as power series, and Cauchy's integral formula.\ See   \cite{Luna3, Charak, JM, Luna1, Luna2, Price, zeta, Ronn} for such results.\ The theory is also useful in applied settings (see\ \cite{Anast, hyperbolic, Rochon, Rochon2}).\ The set of bicomplex numbers $\Bi$ is a four real dimensional extension of the complex numbers $\C$ in the sense that $\Bi \big|_{z_2=0}=\C$.\ Here $\z= x_1+iy_1+jx_2+ijy_2$, where $ij=ji$ so that $(ij)^2=(-1)^2=1$.\ Given $\z$ and $\w=w_1+jw_2$ in $\Bi$,\centre $\z+\w=(z_1+w_1)+j(z_2+w_2)$ and\\ $\z \cdot \w = (z_1w_1-z_2w_2)+j(z_1w_2+w_1z_2)$.\stopcentre These equations agree with complex addition and multiplication when $\z, \w \in \C$, and $\Bi$ is a commutative algebraic ring, as not all bicomplex numbers have a multiplicative inverse.\  But a so-called idempotent representation \centre${\z= (z_1-iz_2)\ee+(z_1+iz_2)\eee \equiv \z_1\ee + \z_2\eee}\,   $\stopcentre importantly gives a better way to think about bicomplex numbers and associated mathematical structures such as analytic functions (see \cite{JM}) and, in particular, linear transformations, as this paper shows. Here $\ee=\displaystyle {(1+ij)/}{2}$ and $\displaystyle \eee={(1-ij)/}{2}$, and $\z_1, \z_2 \in \C$ are the idempotent components.\ The equalities $\ee^2=\ee$, $\eee^2=\eee$, and $\ee\eee = \eee\ee = 0$ result in the following property:\ For $\z= \z_1\ee + \z_2\eee$ and $\w = \w_1\ee + \w_2\eee$, the product \begin{center}$\z \w = \z_1 \w_1 \ee + \z_2 \w_2 \eee$\end{center} computes idempotent componentwise!\ Hence $\z^{n}= \z_1^{n} \ee + \z_2^{n} \eee$, and any (multi-branched) $n$th root translates into complex idempotent components (say for $n \in \mathbb{N}$) as \centre$\z^{1/n}= \z_1^{1/n} \ee + \z_2^{1/n} \eee$.\stopcentre Simple algebra also proves two extremely important special items.\ First, $\z$ is complex if and only if $\z_1=\z_2=\z$.\ Second, the multiplicative noninvertible elements $\z$ are exactly the values that have either $\z_1=0$ or $\z_2=0$.\ 

A powerful so-called hyperbolic norm evolves from consideration of bicomplex numbers with real idempotent components.\ This norm equals the Euclidean norm when applied to complex numbers, but it outputs moduli for noncomplex numbers that are in general not real! Hence its properties form a generalization of the standard defining properties of a norm, obtained by replacing the nonnegative real numbers by a partially ordered set of bicomplex numbers called the nonnegative hyperbolic numbers $\h^+$, which the following definition describes.\ To distinguish it from a norm that has the standard properties, mathematicians sometimes call the hyperbolic norm a poset-valued norm (cf. \cite{Luna1}).

\begin{definition} [Hyperbolic Numbers]
 The hyperbolic numbers \begin{center}$\h=\{x+ijy:\ x,y \in \R\}$\end{center} are a strict subset of $\Bi$.\ The subset of nonnegative hyperbolic numbers is \centre$\h^+ = \{\eta_1\ee + \eta_2 \eee:\eta_1 , \eta_2 \ge 0 \}$.\stopcentre
For any $\z = \z_1\ee + \z_2\eee \in \Bi$, the hyperbolic norm is defined as \centre$|\z |_{_\h}= |\z_1|\ee + |\z_2|\eee$,\stopcentre  which is an element of $\h^+$ and satisfies $|\z \cdot \w|_{_\h}=|\z|_{_\h}\cdot |\w|_{_\h} $ for any $\z,\w \in \Bi$.\  
\end{definition}

A partial ordering on $\h^+$ now follows.\ For two elements $\z=\eta_1\ee+\eta_2\eee$ and $\w=\psi_1\ee+\psi_2\eee$ in $\h^+$, define the hyperbolic inequality $\z <_{_\h} \w$ when both $\eta_1 < \psi_1$ and $\eta_2 < \psi_2$.\ A less-than-or-equal-to partial ordering is similarly defined.\ The hyperbolic norm of any bicomplex number is a hyperbolic number.\ That norm can therefore be used via this hyperbolic inequality to define open balls of bicomplex numbers that have hyperbolic radii.\ The balls are of the form \centre $B_{_\h}(c,R)=\{\z:|\z-c|_{_\h}<_{_\h} R\}$,\stopcentre where the center $c \in \Bi$ and the hyperbolic radius $R \in \h^+$.\ The collection of balls $B_{_\h}(c,R)$ can also be thought of as the neighborhood basis open sets in a topology induced by the hyperbolic norm.

A bicomplex number's Euclidean form $\eta=z_1+jz_2$ produces a choice of three natural conjugations (cf. \cite[p. 562]{nDimensional} or \cite[p. 593]{KumarSingh}):
\begin{enumerate}
  \item $\overline{\z}=\overline{z}_1+j\overline{z}_2$;
  \item $\overline{\z}={z}_1-j{z}_2$; or
  \item $\overline{\z}=\overline{z}_1-j\overline{z}_2$
\end{enumerate}
The first and third choices both generalize conjugation for complex numbers but the third is most natural because it corresponds with the only natural conjugations in terms of the idempotent component; namely, for $\z= \z_1\ee + \z_2\eee$, \centre $\overline{\z}=\overline{z}_1-j\overline{z}_2= \overline{\z}_1\ee + \overline{\z}_2\eee$ \stopcentre This calculation follows from the formulas for $\z_1$ and $\z_2$ in terms of $z_1$ and $z_2$, along with the simple facts that $\overline{\ee}=\ee$ and $\overline{\eee}=\eee$.\ Furthermore, \centre   $\z \cdot \overline{\z}=(\z_1\ee + \z_2\eee)( \overline{\z}_1\ee + \overline{\z}_2\eee)= |\z_1|^2\ee + |\z_2|^2\eee=|\z|^2_{_\h} $,\stopcentre which says this conjugation in its relationship with the hyperbolic norm generalizes complex conjugation in its relationship with the complex modulus. In all that follows in this paper, this third choice is always used as the definition of the bicomplex conjugate, which we formalize here.

\begin{definition} The bicomplex conjugate of $\z= \z_1\ee + \z_2\eee$ is $\overline{\z}=\overline{\z}_1\ee + \overline{\z}_2\eee$.
\end{definition}

Finally, we introduce the notion of a bicomplex vector space, expanding the definition of a vector space on complex scalars. The expansion to the bicomplex setting makes special use of the idempotent representation.

\begin{definition}
A bicomplex vector space, taken over bicomplex scalars, is a set of well-defined elements, called ``bicomplex vectors,'' that may be decomposed using the bicomplex idempotent structure into complex idempotent components as $v=v_1\ee+v_2\eee$, where the resulting corresponding sets of idempotent ``vectors''  $\{v_1\}$ and $\{ v_2 \}$, here taken over complex scalars unless noted otherwise, are each vector spaces with the same dimension.\ The bicomplex vector space $V=\{ v \}$ has finite dimension $n^2$ when $\{ v_1 \}$ and  $\{ v_2 \}$ both have dimension $n$, where a basis for $V$ is formed from all $n^2$ combinations of the bases for the two idempotent components $\{ b_{1,k} \}_{k=1}^n$ and  $\{ b_{2,j} \}_{j=1}^n$, respectively, as 
$\{ b_{k,j}\} = \{ b_{1,k}\ee + b_{2,j}\eee \} $.   
\end{definition}

This paper examines such vector spaces with well-defined complete norms defined on each idempotent component, so that each component vector space forms a Banach space, calling such vector spaces ``bicomplex Banach spaces.'' In addition, it examines such vector spaces with well-defined inner products, so that each component vector space forms a Hilbert space, calling such vector spaces ``bicomplex Hilbert spaces.'' Examining the bicomplex Banach and Hilbert spaces' idempotent structure, it will also define and examine types of linear transformations on such spaces.  

\noindent \section{Bicomplex Matrices}

We first consider the extension of complex linear algebra to linear transformations on finite $n$-dimensional bicomplex vector spaces. In what follows in this section, each matrix is $n \times n$ unless otherwise noted.\ Writing each entry of a bicomplex matrix $A$ in terms of its idempotent representation produces an idempotent representation for the matrix $A=A_1\ee+A_2\eee$. Multiplication of bicomplex matrices is defined via row times column products in the same way as for complex matrices.\ A delightful fact is that matrix multiplication computes idempotent componentwise, just as multiplication computes componentwise for scalars! See, e.g., \cite[Chapter 5]{Luna3}.  This fact is so important that we call it the Fundamental Theorem for Bicomplex Matrices.

\begin{theorem} {\bf The Fundamental Theorem for Bicomplex Matrices:} Given bicomplex matrices $A=A_1\ee+A_2\eee$ of size $m \times k$ and $B=B_1\ee+B_2\eee$ of size $k \times n$, then $A \cdot B = A_1B_1\ee+A_2B_2\eee$ is $m \times n$. 
\end{theorem} 

\begin{proof} Since $\ee^2=\ee$, $\eee^2=\eee$, and $\ee\eee = \eee\ee = 0$,\\ \hspace*{31pt}$A \cdot B = (A_1\ee+A_2\eee)(B_1\ee+B_2\eee)\\  \hspace*{56pt} =A_1B_1\ee^2+A_1B_2\ee\eee+A_2B_1\eee\ee+A_2B_2\eee^2\\ \hspace*{56pt}=A_1B_1\ee+A_2B_2\eee$. \vspace*{4pt} 

\noindent For an equivalent proof, denote $AB=[(ab)_{ij}]$, with $A=[a_{ik}]$ and $B=[b_{kj}]$. Each term is \\ \hspace*{31pt} $(ab)_{ij}=\sum\limits_{k=1}^m a_{ik}b_{kj}=\sum\limits_{k=1}^m (\hat{a}_{ik}\ee +\tilde{a}_{ik}\eee )(\hat{b}_{kj}\ee +\tilde{b}_{kj}\eee)\\ \hspace*{60.5pt} =\sum\limits_{k=1}^m (\hat{a}_{ik}\hat{b}_{kj}\ee +\tilde{a}_{ik}\tilde{b}_{kj}\eee )=\sum\limits_{k=1}^m (\hat{a}_{ik}\hat{b}_{kj})\ee +\sum\limits_{k=1}^m(\tilde{a}_{ik}\tilde{b}_{kj})\eee $.  
\end{proof}

The following three corollaries for $m \times k$ matrices immediately result.\ For the third corollary, we define a bicomplex number $\lambda$ as an eigenvalue for a bicomplex matrix $A$ when there exists a column vector $\vec{v}= \vec{v}_1\ee +\vec{v}_2\eee$ (the corresponding eigenvector) with {\it both} idempotent components $\vec{v}_1$ and $\vec{v}_2$ nonzero so that $A\vec{v} = \lambda \vec{v}$.\footnote{We note that for a matrix $A=A_1\ee+A_2\eee$, a column vector $\vec{v}$ with $A\vec{v} = \lambda \vec{v}$ could have been described as an eigenvector so long as $\vec{v}$ was nonzero, but the specific case of $\vec{v}$ having exactly one nonzero idempotent component reduces to the eigenvalue/eigenvector study of the corresponding matrix component.\ For example, if $\vec{v}_1$ is nonzero but $\vec{v}_2$ is the zero vector, then the examination reduces to the study of the complex matrix $A_1$.\ Our eigenvalue definition matches other authors'; for example, see \cite[Definition 2.3]{nDimensional}.\ A related commentary is \cite[pp. 27--28]{Luna3}.} 

\begin{corollary} For a bicomplex matrix $A=A_1\ee+A_2\eee$, the inverse $A^{-1}$ exists exactly when both $A_1^{-1}$ and $A_2^{-1}$ exist, and then $A^{-1}=A_1^{-1}\ee+A_2^{-1}\eee$.
\end{corollary}

\begin{corollary} Define the determinant of a bicomplex matrix $A$ in terms of its idempotent component determinants; i.e., $\det A \equiv \det A_1\ee +\det A_2\eee$. Then $A^{-1} $ exists exactly when $\det A$ is bicomplex invertible; i.e., when $\det A_1$ and $\det A_2$ are both nonzero.
\end{corollary}

\begin{corollary} A bicomplex number $\lambda$ is an eigenvalue for a bicomplex matrix $A$ exactly when $\det (A - \lambda I)=0$.
\end{corollary} 

The first main result of this section now follows for $n \times n$ bicomplex matrices. 

\begin{theorem} {\bf Bicomplex Jordan Form:} Write $A=A_1\ee+A_2\eee$ and put $A_1$ and $A_2$ in their (complex) Jordan forms as $A_1=P_1J_1P_1^{-1}$ and $A_2=P_2J_2P_2^{-1}$. Then a Jordan form of $A$ is $A=PJP^{-1}$, where $P=P_1\ee+P_2\eee$ and $J=J_1\ee+J_2\eee$. Because the Jordan forms of $A_1$ and $A_2$ are not unique (for example, the diagonal blocks for $J_1$ and $J_2$ may appear in any order), this Jordan form representation of $A$ can take on many different forms.
\end{theorem}

\begin{proof} Express the matrix in its idempotent form:\vspace{3pt}  \\ \hspace*{26.5pt} $A=A_1\ee+A_2\eee=P_1J_1P_1^{-1}\ee+P_2J_2P_2^{-1}\eee\\ \hspace*{39.6pt} =(P_1\ee+P_2\eee)(J_1P_1^{-1}\ee+J_2P_2^{-1}\eee)\\ \hspace*{39.6pt} =(P_1\ee+P_2\eee)(J_1\ee+J_2\eee)(P_1^{-1}\ee+P_2^{-1}\eee)=PJP^{-1}$.
\end{proof}

It is important to realize the Jordan form $J$ is ``almost diagonal.'' Similar to the complex Jordan form, $J$ has zeros in each entry except for possibly the diagonal and the immediate first off-diagonal in the upper-triangular portion.\ The proof of Theorem 2 shows the upper off-diagonal can consist of 0's, 1's, $\ee$'s, and $\eee$'s.\ Also, $A_1$ and $A_2$ are diagonalizable if and only if $J$ is diagonal.\  This paper will soon show an immediate application of the Bicomplex Jordan form: it helps determine the invariant subspaces for any bicomplex matrix. To clarify, when $\Bi^n$ is the collection of $n$-dimensional column vectors with bicomplex entries, a subspace $S \subseteq \Bi^n$ is invariant under an $n \times n$ matrix $A$ when $A\vec{\z} \in S$ for every $\vec{\z} \in S$. The following fact, the second main result of the section, results as a direct consequence of the Fundamental Theorem of Bicomplex Matrices.

\begin{corollary} Writing any vector $\vec{\z} \in \Bi^n$ as $\vec{\z}=\vec{\z}_1\ee+\vec{\z}_2\eee$, where $\vec{\z}_1,\vec{\z}_2 \in \C^n$, decompose any given subspace $S \subseteq \Bi^n$ as $S = S_1\ee \oplus S_2\eee$, where $S_1$ and $S_2$ are subspaces of $\C^n$. Then $S$ is invariant under any given $n \times n$ bicomplex matrix $A=A_1\ee + A_2\eee$ exactly when $S_1$ is invariant under $A_1$ and $S_2$ is invariant under $A_2$. 
\end{corollary} 

\begin{proof} The subspace $S$ is invariant under   $A $ exactly when \centre$A\vec{\z}=A_1\vec{\z}_1\ee+A_2\vec{\z}_2\eee \in S_1\ee \oplus S_2\eee$ for every $\vec{\z} = \vec{\z}_1\ee+\vec{\z}_2\eee \in S$.\stopcentre  But that fact is equivalent to both \centre$A_1\vec{\z}_1 \in S_1$ for every $\vec{\z}_1 \in S_1$,  and $A_2 \vec{\z}_2 \in S_2$ for every $\vec{\z}_2 \in S_2$,\stopcentre which means exactly that $S_1$ is invariant under $A_1$ and $S_2$ is invariant under $A_2$. 
\end{proof}

Corollary 4 is extremely useful, especially since a constructive method to determine the invariant subspace lattice diagram for any complex $n \times n$ matrix is known. See \cite{CJW}.\ From Corollary 4, this constructive method extends to find the invariant subspace lattice diagram for any bicomplex $n \times n$ matrix $A$: it says simply to form each possible invariant subspace $S_1\ee \oplus S_2\eee$, where $S_1$ is invariant for $A_1$ and $S_2$ is invariant for $A_2$, and then to realize each of these invariant subspaces for $A$ ``locks into place naturally'' into the corresponding lattice diagram. Example 2 illustrates that process for a simple case, below. \\

An analysis that is similar to the development of the Jordan form also applies to other representations of bicomplex matrices, for example, to form a diagonalization of a  self-adjoint bicomplex matrix. Begin with the following definitions.

\begin{definition}
The adjoint of a bicomplex matrix $A=[a_{ij}]$ is its conjugate-transpose $A^*=[\overline{a}_{ji}]$. 
\end{definition}

Note this definition generalizes the adjoint of a complex matrix. Furthermore, if $\vec{\z}, \vec{\omega} \in  \Bi^n$,  we define as in, for example, \cite{nDimensional}, a bicomplex inner product as \centre
$\langle \vec{\z}, \vec{\omega} \rangle=\sum\limits_{i=1}^n \z_i\overline{\omega}_i = \sum\limits_{i=1}^n (\z_{i1}\overline{\omega}_{i1}\ee+\z_{i2}\overline{\omega}_{i2}\eee )= \langle \vec{\z}_1, \vec{\omega}_1 \rangle\ee+\langle \vec{\z}_2, \vec{\omega}_2 \rangle\eee$.\stopcentre The following theorem results.

\begin{theorem} For any $n \times n$ bicomplex matrix $A$ and $\vec{\z},\vec{\omega} \in \Bi^n$, \centre $\langle A\vec{\z}, \vec{\omega} \rangle=\langle \vec{\z}, A^*\vec{\omega} \rangle$.\stopcentre
\end{theorem} 

\begin{proof} By Theorem 1, the definition of the bicomplex inner product, and the adjoint properties of complex matrices, \\ \indent  $\langle A\vec{\z}, \vec{\omega} \rangle=\langle A_1\vec{\z}_1\ee+ A_2\vec{\z}_2\eee, \vec{\omega}_1\ee+\vec{\omega}_2\eee \rangle=\langle A_1\vec{\z}_1, \vec{\omega}_1\rangle\ee+ \langle A_2\vec{\z}_2, \vec{\omega}_2\rangle\eee  \\ \indent  \hspace*{37.5pt} = \langle \vec{\z}_1, A_1^*\vec{\omega}_1\rangle \ee+\langle \vec{\z}_2, A_2^*\vec{\omega}_2\rangle\eee   =\langle \vec{\z}, A^*\vec{\omega} \rangle$
\end{proof}

Similar to the complex situation, a bicomplex matrix is self-adjoint when $A=A^*$. A bicomplex unitary matrix $U$ satisfies $U^{-1}=U^*$.\ We note, as in \cite{nDimensional}, the following well-known result. 

\begin{theorem} {\bf The Spectral Theorem for Self-adjoint Matrices:} Any self-adjoint bicomplex matrix is unitarily diagonalizable, in the sense that \centre $A = PDP^{-1}$ \stopcentre  for a (bicomplex) diagonal matrix $D$ and a unitary matrix $P$. 
\end{theorem}

\begin{proof} When $ A=A_1\ee+ A_2\eee$, by the definition of a self-adjoint bicomplex matrix, both $A_1$ and $A_2$ are complex self-adjoint, and then the spectral theorem implies\\ $A_1=P_1D_1P_1^{-1}$ and $A_2=P_2D_2P_2^{-1}$. Then \vspace*{5pt} \\ \indent \hspace*{17pt} $A=A_1\ee+ A_2\eee= P_1D_1P_1^{-1}\ee + P_2D_2P_2^{-1}\eee\\ \hspace*{40pt} =(P_1\ee+P_2\eee)(D_1P_1^{-1}\ee + D_2P_2^{-1}\eee)\\ \indent \hspace*{28 pt} =(P_1\ee+P_2\eee)(D_1\ee+D_2\eee)(P_1^{-1}\ee + P_2^{-1}\eee)=PDP^{-1}$, \vspace*{5pt} \\  with each matrix having the properties required in the theorem statement.\ For example, $P$ may be chosen unitary because $P_1$ and $P_2$ may be chosen as (complex) unitary, and then \centre$P^*=P_1^*\ee+P_2^*\eee=P_1^{-1}\ee+P_2^{-1}\eee=P^{-1}$.\vspace{-19pt}\stopcentre
\end{proof}

\noindent \section{Examples}

\noindent {\bf Example 1:} Let $A=\left[\begin{matrix} 0 & \z \\ \overline{\z} & 0 \end{matrix}\right]=\left[\begin{matrix} 0 & \z_1 \\ \overline{\z_1} & 0 \end{matrix}\right]\ee +\left[\begin{matrix} 0 & \z_2 \\ \overline{\z_2} & 0 \end{matrix}\right]\eee$ with $\z \ne 0$. Then $\det (A- \lambda I) =0$ implies $\lambda_1^2 -|\z_1|^2 =0$, or $\lambda_1=\pm|\z_1|$,  and  similarly $\lambda_2=\pm|\z_2|$. Therefore, $A$ can be analyzed in terms of the pair of bicomplex eigenvalues $\lambda = |\z_1|\ee+|\z_2|\eee, -|\z_1|\ee-|\z_2|\eee$, more concisely written as $\lambda = |\z |_{_{\h}}, - |\z |_{_{\h}}$. The corresponding eigenvectors (which come as the nonzero solutions to $A\vec{\z}=\lambda \vec{\z}$) are
\centre$\left[ \begin{matrix} |\z |_{_{\h}} \\ \overline{\z} \end{matrix} \right]$ and $\left[ \begin{matrix} -|\z |_{_{\h}} \\ \overline{\z} \end{matrix} \right]$, respectively.\stopcentre  The Spectral Theorem for Self-adjoint Matrices then implies 
$A = PDP^{-1}$, where \centre   $P= \left[ \begin{matrix} |\z |_{_{\h}} &   -|\z |_{_{\h}}  \\ \overline{\z} & \overline{\z}  \end{matrix} \right]$  and $D=\left[ \begin{matrix} |\z |_{_{\h}} & 0 \\ 0 &   -|\z |_{_{\h}} \end{matrix} \right]$.  \stopcentre An important note is that this is the standard complex Jordan form for $A$ when $\z \in \mathbb{C}$. 

But even when $A$ is complex, these bicomplex diagonalizations are generally not unique (nor even ``essentially'' unique). In this example, a second diagonalization results, realizing $\det (A- \lambda I) =0$ also for the pair of eigenvalues $\lambda = |\z_1|\ee-|\z_2|\eee, -|\z_1|\ee+|\z_2|\eee$. The corresponding (nonzero) eigenvectors are 
\centre$\left[ \begin{matrix} |\z |_1\ee - |\z_2 |\eee \\ \overline{\z} \end{matrix} \right]$ and $\left[ \begin{matrix} -|\z |_1\ee + |\z_2 |\eee\\ \overline{\z} \end{matrix} \right]$, respectively.\stopcentre A second diagonalization provided by the Spectral Theorem for Self-adjoint Matrices results, even when $A$ is complex. We get $A = PDP^{-1}$, where \centre   $P= \left[ \begin{matrix} |\z |_1\ee - |\z_2 |\eee &   -|\z |_1\ee + |\z_2 |\eee \\ \overline{\z} & \overline{\z}  \end{matrix} \right]$  and $D=\left[ \begin{matrix} |\z_1|\ee-|\z_2|\eee & 0 \\ 0 &   -|\z_1|\ee+|\z_2|\eee  \end{matrix} \right]$.  \stopcentre 
In general, an $n \times n$ self-adjoint matrix $A=A_1\ee+A_2\eee$ will have $n!$ different bicomplex diagonalizations when $A_1$ and $A_2$ have $n$ distinct eigenvalues, since these eigenvalues can be paired in $n!$ ways. 
 
\bigskip

\noindent {\bf Example 2:} {\bf An Illustration of the Invariant Subspace Lattice Diagram for Bicomplex Matrices.} Let $A= \left[\begin{matrix} 0 & 0 \\ 0 & 0 \end{matrix}\right]\ee +\left[\begin{matrix} 0 & 1 \\ 0 & 0 \end{matrix}\right]\eee$. This matrix $A$ is a bicomplex Jordan form matrix as Theorem 2 describes. As developed in \cite{CJW}, the invariant subspace lattice diagram for a complex matrix can be determined in a fairly simple manner from a complex matrix's Jordan form structure. For example, in the simple case for a matrix with invariant subspaces that are all marked\footnote{An invariant subspace $\mathcal{N}$ is ``marked'' when there exists a Jordan basis for $\mathcal{N}$ that can be extended (by adjoining new vectors) to form a Jordan basis for the entire space $\mathbb{C}^n$  (see \cite[p. 210]{Rodman}).} and whose Jordan form has a single eigenvalue (see \cite[Theorem A1]{CJW}), the invariant subspace lattice diagram is found by forming diagonal block matrices $Z$ with the same block structure as the Jordan form $J$ but with exponential powers of the backward shift matrix forming the diagonal blocks. Choosing a set of exponential powers corresponds to a particular invariant subspace (in a pairwise manner).  Creating a lattice diagram with all the different $Z$'s formed in this manner thereby creates a lattice diagram equivalent to the invariant subspace's lattice diagram. 

Corollary 4 implies the invariant subspace lattice diagram for a bicomplex matrices $A=A_1\ee+ A_2\eee$ can be found from combining the invariant subspace lattices, applying the techniques in \cite{CJW} on the complex matrices $A_1$ and $A_2$. In this example, $A_1=\left[\begin{matrix} 0 & 0 \\ 0 & 0 \end{matrix}\right]$ has two Jordan blocks with eigenvalues 0, and $A_2= \left[\begin{matrix} 0 & 1 \\ 0 & 0 \end{matrix}\right]$ has a single Jordan block with eigenvalues 0. Their two invariant subspace lattice diagrams, found using the techniques in \cite{CJW}, are in Figure 1.
\begin{figure}[h] 
\centering 
  \includegraphics[bb=0 0 701 252,width=5.14in,height=1.85in,keepaspectratio]{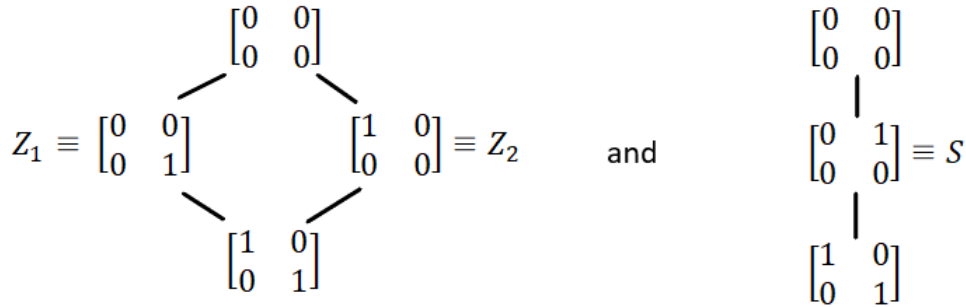}
    \caption{Invariant subspace lattices for complex matrices $A_1$ (on left) and $A_2$ (on right)}
  \label{fig:Figure1}
\end{figure}

Combining the invariant subspaces from each idempotent portion creates an invariant subspace lattice diagram for the matrix $A$. Our diagram uses a notational convenience: the symbol $[Z_1 | S]$, for example, means that the invariant subspace for $A $ is formed from the invariant subspaces that come from each matrix in the idempotent component.\ Here, since  $Z_1$ corresponds to the invariant subspace span$\{\vec{e}_1\}$, where $\vec{e}_1$ is the eigenvector that corresponds to the upper left Jordan block of $A_1$ (see \cite{CJW} for this formulation), and $S$ corresponds to the invariant subspace span$\{\vec{e}_2\}$, where $\vec{e}_2$ is the eigenvector for the single Jordan block of $A_2$, the symbol $[Z_1 | S]$ corresponds to the invariant subspace span$\{\vec{e}_1\}\ee+\mbox{span}\{\vec{e}_2\}\eee$ of $A$. With this symbolism in mind, the invariant subspace lattice diagram for $A$ is formed from all such combinations as follows in Figure 2. 
\begin{figure}[h] 
  \centering
  \includegraphics[width=5.14in,height=2.01in,keepaspectratio]{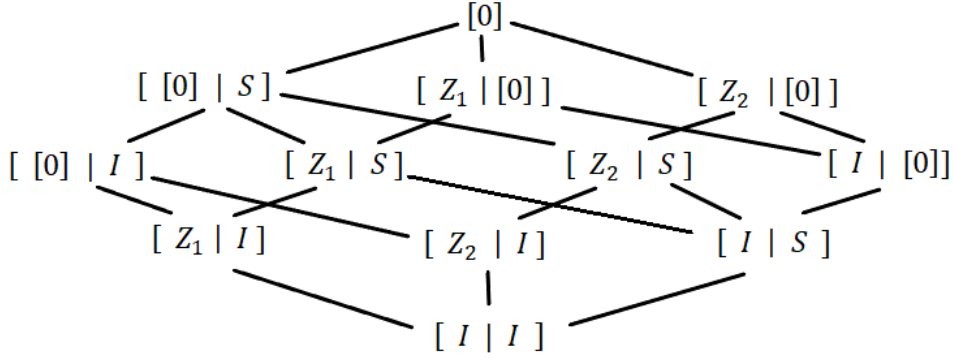}
  \caption{Invariant Subspace Lattice for the bicomplex matrix $A$}
  \label{fig:Figure2}
\end{figure}

The same methodology, using the techniques of \cite{CJW} on each idempotent component of any $n \times n$ bicomplex matrix  $A=A_1\ee+A_2\eee$, routinely develops the bicomplex matrix's complete invariant subspace lattice diagram. 

\section{Bicomplex Compact Operators on Banach Spaces}

This section shows how the theory of operators on complex Banach spaces extends, via an operator's idempotent representation, to a corresponding theory of so-called bicomplex compact operators on bicomplex Banach spaces.

\begin{definition} A bicomplex Hilbert space is of the form $\mathcal{H}=\mathcal{H}_1\ee+\mathcal{H}_2\eee$, where $\mathcal{H}_1$ and $\mathcal{H}_2$ are complex Hilbert spaces with the same dimension.\ The bicomplex inner product on any elements $\vec{\z}, \vec{\omega} \in \mathcal{H}$   is defined via complex inner products on the idempotent components of $\vec{\z}$ and $\vec{\omega}$ as \centre $\langle \vec{\z}, \vec{\omega} \rangle= \langle \vec{\z}_1, \vec{\omega}_1 \rangle\ee+\langle \vec{\z}_2, \vec{\omega}_2 \rangle\eee$ \stopcentre A set of bicomplex vectors $\{ \vec{\z}_k \}$ is orthogonal in $\mathcal{H}$ when $\langle \vec{\z}_n, \vec{\z}_m \rangle=0$ for $n \ne m$ and orthonormal when in addition the hyperbolic norm of each vector in the set, defined via idempotent component norms for $\mathcal{H}_1$ and $\mathcal{H}_2$ according to  $\| \vec{\z}_n \|_{_\h} = \| \vec{\z}_{n,1} \|\ee + \| \vec{\z}_{n,2} \|\eee$, equals 1.

Similarly, a bicomplex Banach space is of the form $\mathcal{B}=\mathcal{B}_1\ee+\mathcal{B}_2\eee$, where $\mathcal{B}_1$ and $\mathcal{B}_2$ are complex Banach spaces with the same dimension. 
\end{definition}

Such spaces are often called ``$\Bi$-modules'' in the literature (see, for example, \cite{KumarSingh}), where they have been studied in detail.\ Definition 5 appropriately names such a space as a Hilbert and/or Banach space, as the idempotent arrangement is uniquely structured to produce broad general theories for linear operators on such spaces. This and the next section develops such theory for bicomplex compact operators, defined here:

\begin{definition} An operator $K:\ \mathcal{B} \to \mathcal{B}$ mapping a bicomplex Banach space $\mathcal{B}$ to itself is compact if for every bounded set $S \subseteq \mathcal{B}$, the set $K(S)$ is relatively compact in $\mathcal{B}$, which means that every infinite subset of $K(S) \equiv \{ K\vec{s}: \vec{s}\in S \}$ has at least one limit point in its (hyperbolic) norm closure.\ This is equivalent to every bounded sequence $\{x_n\}$ of vectors in $\mathcal{B}$ producing the sequence $\{Kx_n\}$ with a hyperbolic norm-convergent subsequence. 
\end{definition}

Note importantly that hyperbolic norm convergence in $\mathcal{B}$ happens exactly when norm convergence of the idempotent components happen in both complex spaces $\mathcal{B}_1$ and $\mathcal{B}_2$.\ In such a way, the actions of any operator in $\mathcal{B}$ follows from actions of the corresponding operator elements acting on $\mathcal{B}_1$ and $\mathcal{B}_2$.\ Here, $T$ acts on $v=v_1\ee+v_2\eee \in \mathcal{B}$ according to $Tv=T_1v_1\ee+T_2v_2\eee$.\ The following lemma notes the impact of those facts on compact operators.

\begin{lemma} For any operator $T=T_1\ee+T_2\eee$  on a bicomplex Banach space $\mathcal{B}=\mathcal{B}_1\ee+\mathcal{B}_2\eee$, the operators $T_1$ and $T_2$ are compact on $\mathcal{B}_1$ and $\mathcal{B}_2$ respectively iff $T$ is compact.
\end{lemma}

\begin{proof} For $i=1,2$, the compactness of $T_i$ implies for every bounded set $S_i \subseteq \mathcal{B}_i$, the set $T_i(S_i)$ is relatively compact in $\mathcal{B}_i$; i.e., every infinite subset of $T_i(S_i)$ has at least one limit point $x_i$ in its (complex Banach space) norm closure.\ But then every hyperbolic bounded set $S=S_1\ee + S_2\eee$, which must have $S_1$ and $S_2$ bounded in the corresponding complex Banach space norm, has at least one limit point $x=x_1\ee+x_2\eee$ in its (hyperbolic) norm closure. Hence $T$ is compact by definition.

Conversely, assume $T=T_1\ee+T_2\eee$ is compact and take any bounded sequences $\{x_{1,n}\}\in  \mathcal{B}_1$ and $\{x_{2,n}\} \in  \mathcal{B}_2$. Then $\{x_n\}=\{x_{1,n}\ee+x_{2,n}\eee\}$ is a bounded set of vectors in $\mathcal{B}$ (in the hyperbolic norm sense),  producing the sequence $\{Tx_n\}$ with a hyperbolic norm-convergent subsequence $\{Tx_{{n}_k}\equiv T_1x_{{1,n}_k}\ee+T_2x_{{2,n}_k}\eee \}$.\ But the hyperbolic norm structure then implies that each idempotent component $T_1x_{{1,n}_k}$ and $T_2x_{{2,n}_k}$ is norm-convergent in $\mathcal{B}_1$ and $\mathcal{B}_2$, respectively.\ Hence $T_1$ and $T_2$ each satisfy the definition of a compact operator.
\end{proof}

The (hyperbolic) operator norm of any (bicomplex) bounded operator $T$ is naturally defined as $\|T\|_{_\h}=\|T_1\|\ee+\|T_2\|\eee$.\   Also, for example, it is easy to see that the identity operator $I$ is not compact on  any infinite-dimensional bicomplex Hilbert space $\mathcal{H}$.\ For we can choose an infinite sequence $x_n$ of $\mathcal{H}$ to be the orthonormal set of basis elements for the Hilbert space, which are each distance (in hyperbolic norm) 1 apart. Then $I(x_n)=x_n$ has no convergent subsequence.\ A similar result is seen true for the identity operator on any infinite-dimensional bicomplex Banach space by applying Lemma 2 below. 

The spectrum of a bicomplex operator is defined similarly to the spectrum of a complex operator.

\begin{definition}
For a bounded operator $T$:
\begin{enumerate} 
  \item The spectrum of $T$ is $\sigma(T)= \{ \zeta \in \Bi\  \colon  \zeta I - T \mbox{ is not invertible } \}$.
  \item The point spectrum of $T$ is the set of eigenvalues $\sigma_p(T)$; in other words, the set of $\lambda \in \Bi$ that satisfy \centre $(\lambda I - T)v=0$ \stopcentre for some $v=v_1\ee+v_2\eee \in \mathcal{B}$ with both $v_1$ and $v_2$ nonzero, where $v$ is called an eigenvector for $\lambda$. Any eigenvalue $\lambda=\lambda_1\ee+\lambda_2\eee$ produces for $T=T_1\ee+T_2\eee$ the eigenvalue $\lambda_1$ for $T_1$ and $\lambda_2$ for $T_2$ on the associated spaces $\mathcal{B}_1$ and $\mathcal{B}_2$, respectively, with associated eigenvectors $v_1$ and $v_2$, respectively.
\end{enumerate}
\end{definition}

A spectral theory for bicomplex compact operators holds on bicomplex Banach spaces.\ The following theorem describes that fact, and is well known in the case of compact operators on complex Banach spaces (see, for example 
\cite[pp. 377--378]{Johnston}). 

\begin{theorem}  For a bicomplex compact operator $K:\mathcal{B} \to \mathcal{B}$ with $\mathcal{B}$ a bicomplex Banach space,
\begin{enumerate} 
\item Every invertible $\lambda \in \sigma(K)$ is an eigenvalue of $K$.

\item The eigenvalues can only accumulate at noninvertible values.\ If the dimension of $\mathcal{B}$ is not finite, then $\sigma(K)$ must contain 0, even if 0 is not an eigenvalue.

\item $\sigma(K)$ is countable; in other words, its elements can be listed as either a finite list or as $\sigma(K)=\{\lambda_1,\lambda_2,\ldots \}$.

\end{enumerate} 
\end{theorem}

The theorem's proof follows from the fact that $K$ on $\mathcal{B}$ decomposes into $K=K_1\ee+K_2\eee$, where $K_i$ is compact on $\mathcal{B}_i$, $i=1,2$.\  Applying the fact that the theorem is well known and true on complex Banach spaces, the proof of each part follows immediately. 

Furthermore, the next two lemmas, which are historically important and true for operators on complex Banach spaces, are true for bicomplex Banach spaces as well.\ The first uses the concept of an infimum of a set of hyperbolic numbers $\eta=\eta_1\ee+\eta_2\eee$ for which $\eta_1, \eta_2 \in \mathbb{R}$, defined as $\inf \{ \eta \}=\inf \{ \eta_1\} \ee + \inf \{ \eta_2 \} \eee$. \bigskip

\begin{lemma} {\bf Riesz's Lemma.}\ For a non-dense subspace $X=X_1\ee+X_2\eee$ of a bicomplex Banach space $\mathcal{B}=\mathcal{B}_1\ee+\mathcal{B}_2\eee$ , given $0<r < 1$, there is an element $y \in  \mathcal{B}$ with $\|y\|_{\h} = 1$ and $\inf\limits_{x \in X}\|x - y\|_{\h}= r$.
\end{lemma}

\begin{proof} As mentioned, the lemma holds for each of the associated operators on the complex Banach spaces: $X_1$ in $\mathcal{B}_1$ and $X_2$ in $\mathcal{B}_2$, in the sense that it produces a vector $y_1 \in \mathcal{B}_1$ and $y_2 \in \mathcal{B}_2$ with $\|y_i\| = 1$ and $\inf\limits_{x \in X_i}\|x - y_i\|= r$ for $i=1,2$.\  Then the vector $y=y_1\ee+y_2\eee$ satisfies the lemma, since  $\|y\|_{\h} = 1\ee+1\eee=1$ and $\inf\limits_{x \in X}\|x - y\|_{_\h}= r\ee+r\eee=r$. 
\end{proof}

\begin{lemma}  For a compact operator $K=K_1\ee+K_2\eee$  on a bicomplex Banach space $\mathcal{B}=\mathcal{B}_1\ee+\mathcal{B}_2\eee$ and for which $I-K$ is one-to-one, the range of $I-K$ is closed in the hyperbolic norm.
\end{lemma}

\begin{proof} The lemma holds for operators on a complex Banach space; see, for example, \cite{Garrett}. Apply that fact to each of $K_1$ on $\mathcal{B}_1\ee$ and $K_2$ on $\mathcal{B}_2\eee$, producing, since $I-K_1$ must be one-to-one on $\mathcal{B}_1\ee$ and $I-K_2$ must be one-to-one on $\mathcal{B}_2\eee$, closed ranges for $I-K_1$ and $I-K_2$. But then the range of $I-K$ is closed as desired. 
\end{proof}

\section{Bicomplex Compact Operators on Hilbert Spaces}

Similar as for the material in the last section, important theorems describe the structure of a bicomplex compact operator $K$ on a bicomplex Hilbert space $\mathcal{H}=\mathcal{H}_1\ee+\mathcal{H}_2\eee$. The theorems result from corresponding theory for each complex compact operator $K_1$ and $K_2$ that form the idempotent components of $K$ and act on the complex Hilbert spaces $\mathcal{H}_1$ and $\mathcal{H}_2$, respectively.\  

\begin{theorem} For any operator $T=T_1\ee+T_2\eee$  on a bicomplex Hilbert space $\mathcal{H}=\mathcal{H}_1\ee+\mathcal{H}_2\eee$, if there is a sequence $\{K_n \}_{n=1}^\infty$ of bicomplex compact operators with $\| T - K_n \|_{_\h} \to 0$, then $T$ is compact.
\end{theorem}

\begin{proof} Since the (hyperbolic) operator norm of any bicomplex operator $Y$ is calculated via idempotent complex operator norms as $\| Y \|_{_\h} =\| Y_1\|\ee+\|Y_2 \|\eee$, there are two (idempotent component) sequences of complex compact operators $K_{n,i}$ such that  $\| T_i - K_{n,i} \| \to 0$ for $i=1,2$.\  Then, since the theorem is true for any operator on a complex Hilbert space (cf. \cite[p. 367]{Johnston}), it must be the case that the idempotent components $T_1$ and $T_2$ are both compact.\ But then $T$ is compact by Lemma 1. 
\end{proof}

Since every finite-rank bicomplex operator in $\mathcal{B}(\mathcal{H})$ (of the form \centre $T(f) = \sum\limits_{i=1}^m \sigma_i \langle f, g_i \rangle g_i $\stopcentre  for hyperbolic nonnegative bicomplex $\sigma_i$ and $g_i$   chosen as orthogonal in $\mathcal{H}$) is compact (as it splits into compact operators on $\mathcal{H}_1$ and $\mathcal{H}_2$), Theorem 6 says any bicomplex operator norm limit of a bicomplex finite-rank operator is compact. But just as for complex compact operators on complex Hilbert spaces, the converse is also true, as the next theorem points out. See \cite[p. 11]{Charak2} for a parallel statement of the next theorem.

\begin{theorem} Every compact operator $K$ on a bicomplex Hilbert space $\mathcal{H}$ is an operator norm limit of finite-rank bicomplex operators.\ In fact, $K$ takes the form \centre  $K(f)=\sum\limits_{i=1}^\infty \sigma_i \langle f, g_i \rangle g_i $\stopcentre  for hyperbolic nonnegative bicomplex $\sigma_i$, and with $g_i$   chosen as orthogonal in $\mathcal{H}$.\ 
\end{theorem}

\begin{proof} Write $K=K_1\ee+K_2\eee$ in its idempotent complex Hilbert space compact operator components, and express both $K_1$ and $K_2$ as operator norm limits (in each idempotent complex Hilbert space) of finite-rank operators:
\centre  $K_j(f)=\sum\limits_{i=1}^\infty \sigma_{i,j} \langle f, g_{i,j} \rangle g_{i,j} $, for $j=1,2$.  \stopcentre
Note $\sigma_{i,j}$ may be chosen as nonnegative and each set of vectors $\{g_{i,j}\}_{i=1}^\infty$ may be chosen as an orthogonal family of vectors for $j=1,2$.\ Recombining the idempotent structures gives
\centre  $K(f)=\sum\limits_{i=1}^\infty \sigma_i \langle f, g_i \rangle g_i $,  \stopcentre
where $\sigma_i= \sigma_{i,1}\ee+ \sigma_{i,2}\eee$ are nonnegative hyperbolic numbers and the set of vectors $\{g_i= g_{i,1}\ee+ g_{i,2}\eee\}_{i=1}^\infty$ forms an orthogonal family in the bicomplex inner product structure. 
\end{proof}

\section{Conclusion}

This paper's results continue analysis found already in many other articles that show how natural it is to use the idempotent structure on both the vector elements of a Hilbert or Banach space as well as the idempotent structure of operators. The idempotent approach, as opposed to misguided efforts that continue to examine the Euclidean rectangular coordinates, easily extends results in linear algebra and operator theory for complex Hilbert or Banach spaces to similar results on bicomplex spaces.\ This paper's theory decidedly confirms that a full study of bicomplex operators is possible as a complete generalization of operator theory on complex operators when using the idempotent structure. The authors urge continued investigations into bicomplex spaces and bicomplex function theory, using the idempotent structures and hyperbolic norms to extend known complex-valued results. 




\end{document}